\documentclass[reqno]{amsart}

\usepackage[foot]{amsaddr}
\usepackage{graphicx,nicefrac}
\usepackage{algorithm,algorithmicx,algpseudocode}
\usepackage{ulem}
\usepackage{cite}
\usepackage{scalerel}
\usepackage{bbding}
\usepackage{caption} 
\usepackage{lmodern}
\usepackage{slantsc}
\usepackage{amstext, amsfonts, amsmath, amssymb, mathrsfs}

\algrenewcommand\algorithmicindent{1em}

\newcommand{\norm}[1]{\left\|#1\right\|}
\newcommand{\F}{\mathsf{F}}
\newcommand{\G}{\mathsf{G}^{\text{\tiny{n}}}}

\newcommand{\A}{\mathsf{A}}
\newcommand{\R}{\mathbb{R}}

\newcommand{\PP}{\mathcal{P}}

\newcommand{\abs}[1]{\left|#1\right|}
\renewcommand{\P}{\tt Poinc}

\newtheorem{theorem}{Theorem}[section]
\theoremstyle{definition}
\newtheorem{lemma}[theorem]{Lemma}
\newtheorem{example}[theorem]{Example}
\newtheorem{remark}[theorem]{Remark}

\title[Linearized Continuous Galerkin \emph{\MakeLowercase{\textbf{hp}}}-FEM]
{Linearized Continuous Galerkin \emph{\MakeLowercase{\textbf{hp}}}-FEM Applied to Nonlinear Initial Value Problems}

\author[M.~Amrein]{Mario Amrein}
\address{Applied University of Zurich, CH-8401 Switzerland}
\email{mario.amrein@zhaw.ch}

\begin{document}
\normalem
\begin{abstract}
In this note we consider the continuous Galerkin time stepping method of arbitrary order as a possible discretization scheme of nonlinear initial value problems. In addition, we develop and generalize a well known existing result for the discrete solution by applying a general linearizing procedure to the nonlinear discrete scheme including also the \emph{simplified} Newton solution procedure. In particular, the presented existence results are implied by choosing sufficient small time steps locally. Furthermore, the established existence results are independent of the local approximation order. Moreover, we will see that the proposed solution scheme is able to significantly reduce the number of iterations. Finally, based on existing and well known \emph{a priori} error estimates for the discrete solution, we present some numerical experiments that highlight the proposed results of this note.
\end{abstract}



\maketitle

\section{Introduction}
In this note we consider the continuous Galerkin (cG) method of arbitrary order applied to the---possibly---nonlinear initial value problem given by 
\begin{equation}
\label{eq:initial}
\begin{cases}
\dot{u}(t)&=\F(t,u(t)),\\
u(0)&=u_{0}.
\end{cases}
\end{equation}
Here, for a final time $T\in (0,\infty)$, $u:(0,T)\rightarrow \mathbb{R}^{d}, d\in \mathbb{N}_{\geq 1}$ signifies the unknown solution and $u_{0}$ is the initial data that determines $u$ at time $t=0$. In addition, $\F:[0,T]\times\mathbb{R}^{d}\rightarrow \mathbb{R}^{d}$ is a---possibly---nonlinear function. It is well known that for $\F$ being continuous, the local existence of a solution is implied by Peano's Theorem (see, e.g., \cite{markley2004principles}). Moreover, in case of $\F$ being---locally---Lipschitz, the solution is even unique by the Theorem of Picard and Lindel\"of (see again, e.g., \cite{markley2004principles}). In general, problem \eqref{eq:initial} can only be solved approximately, i.e. numerically. Such a numerical solution scheme relies mainly on two different procedures which we address in this note: Firstly, the problem at hand needs to be discretized over some finite dimensional subspaces of the solution space. This leads to a series of nonlinear systems that need to be solved again numerically. Thus in a second step, the numerical solution procedure of the nonlinear systems can again only be solved approximately using a---suitable---linearization scheme applied to the nonlinear systems. Typically, such a linearization scheme is given by Banach's fixed point iteration procedure---also termed \emph{Picard} iteration---. Here we recall that---in case of $\F$ being Lipschitz continuous---the proof of the local well posedness of problem \eqref{eq:initial}---by the famous result of Picard and Lindel\"of---is constructive and relies mainly on Banach's fixed point theorem. It is noteworthy that this result can also be achieved constructively using more general iteration schemes. Indeed, with the aim of proving local well posedness of \eqref{eq:initial}, in \cite{deuflhard2011newton} for example, the standard \emph{Picard} iteration procedure is replaced by the \emph{simplified Newton} iteration. Of course in this case, the assumption of $\F$ being Lipschitz continuous needs to be replaced by stronger assumptions, mainly on the derivative of $\F$. However, from a computational point of view a benefit of such an approach that uses more advanced iteration schemes---as for example general \emph{Newton-type} iteration schemes---, should be given by a lower number of iteration steps within the applied numerical solution procedure. 

In this note, the underlying discretization scheme for the approximation of \eqref{eq:initial} is given by a continuous Galerkin time stepping method of arbitrary order. The idea of such an approximation scheme relies essentially on a weak formulation of \eqref{eq:initial}. Subsequently, this weak formulation will be restricted on some finite dimensional subspaces in order to end up with a discretization of \eqref{eq:initial}. In particular, since the test space of the weak formulation can be chosen of polynomials that are discontinuous at the nodal points, this discretization scheme can be interpreted as an \emph{implicit} one-step scheme, see e.g. \cite{Wihler:05,Holm:2018,bangerth2013adaptive,thomee1984galerkin} for further details. Thus, starting from the initial data $u_{0}$, each time step implies a nonlinear system that needs to be solved iteratively. Thus, if the underlying continuous problem \eqref{eq:initial} is well posed for some $T>0$, it is reasonable that a suitable iterative solution procedure---resolving the nonlinearity---is itself well posed as long as the time steps are sufficiently small. Indeed, as we will see, the proposed analysis for the well posedness of the discrete version of \eqref{eq:initial} implied by the continuous Galerkin methodology depends solely on the \emph{local} time steps and is independent of the local approximation order, i.e. the \emph{local} polynomial degree of the solution space.

\subsubsection*{Outline.} The outline of this work is as follows: In Section \ref{sec:2} we present the continuous Galerkin (cG) time stepping scheme used for the discretization of the underlying initial value problem \eqref{eq:initial}. Subsequently, we introduce the proposed iterative linearization scheme used for the solution of the nonlinear discrete system. The purpose of Section \ref{sec:3} is the well posedness of the discretized and now linearized problem. This will be accomplished by our main result given in Theorem \ref{theo:main}. Since the proof of this result relies on a fixed point iteration argument, we end up with an iterative solution procedure that can be tested on some numerical experiments in Section \ref{sec:4}. We further discuss the $h$ and $p$ version of the cG methodology and compare the number of iterations between the standard Banach fixed point iteration procedure and the \emph{simplified} Newton iteration scheme. Finally, we summarize and comment on our findings in Section \ref{sec:concl}. 	

\subsubsection*{Notation.} Throughout this article $(\cdot,\cdot)$ is the Euclidean product of $\R^{d}, d\geq 1$ with the induced norm $\norm{x}=\sqrt{(x,x)}$. For an interval $I=(a,b)$ we denote by $L^2(I;\R^d)$ the usual space of square integrable functions on $I$ with values in $\R^{d}$ and norm $\norm{\cdot}_{L^{2}(I;\R^{d})}$. The set $ L^{\infty}(I;\R^{d})$ is the usual space of bounded functions with norm $\norm{x}_{L^{\infty}(I;\R^d)}=\text{ess}\sup_{t\in I}{\norm{x(t)}}$. In addition, let $H^{s}(I;\mathbb{R}^{d}), s\in \mathbb{N}$ be the standard Sobolev space with corresponding norm $\norm{\cdot}_{H^{s}(I;\R^{d})}$. For a Banach space $X$ we signify by $X'$ the dual of $X$. In addition, $\left\langle \cdot ,\cdot \right\rangle $ will be used for the dual pairing in $X'\times X$, i.e. the value of $x'\in X'$ at the point $x\in X$. We further assume that the right side $\F$ in \eqref{eq:initial} is Lipschitz continuous with respect to the second variable, i.e.
\begin{equation}
\label{eq:Lipschitz}
\norm{\F(t,u)-\F(t,w)}\leq L\norm{u-w}, \quad u,w\in\R^d, \quad t\in [0,T].
\end{equation}

\section*{The \upshape{\textbf{hp}}-\scshape{cG Time Stepping Methodology}}
\label{sec:2}
Let $0=t_{0}<t_{1}<t_{2}<\ldots <t_{N}=T$ be a partition $\tau_{N}=\{I_{n}\}_{n=1}^{N}$ of the interval $I=[0,T]$ into $N$ open sub-intervals $I_{n}=(t_{n-1},t_{n})$. By $k_{n}=t_{n}-t_{n-1}$ we denote the local time steps with $k=\max_{n\in \{1,\ldots,N\}}\{k_{n}\}$ and $r_{n}\in \mathbb{N}_{\geq 0}$ represents the local polynomial degree, i.e. the local approximation order. Furthermore, the space of all polynomials of degree $r\in \mathbb{N}_{\geq 0}$ will be given by
\[
\PP^{r}(I;\R^{d})=\left\{p\in C^{0}(I;\mathbb{R}^{d})|p(t)=\sum_{j=0}^{r}{a_{j}x^{j}}, a_{j}\in \R^{d} \right\}.
\]
We further introduce the vector $\textbf{r}=\{r_{n}\}_{n=1}^{N}$ in order to allocate the local approximation orders. The following approximation 
\[
\mathcal{S}^{\textbf{r},1}(\tau_{N};\mathbb{R}^{d})=\{x\in H^{1}(I,\mathbb{R}^{d})|x_{|I_{n}}\in \PP^{r_{n}}(I_{n};\mathbb{R}^{d}),1\leq n\leq N\}
\]
and
\[
\mathcal{S}^{\textbf{r}-1,0}(\tau_{N};\mathbb{R}^{d})=\{x\in L^{2}(I,\mathbb{R}^{d})|x_{|I_{n}}\in \PP^{r_{n}-1}(I_{n};\mathbb{R}^{d}),1\leq n\leq N\},
\]
test spaces will be used in the sequel. Notice that since the approximation space consists of continuous elements we need to choose a discontinuous test space (as can be seen by the fact that the order of the test space is $\textbf{r}-1$).

\subsection{Discretization.}
The $hp$-cG time stepping scheme---first introduced in \cite{Hulme:72a,Hulme:72,Estep:95,EstepFrench:94}---now reads as follows: Find $\mathcal{U}\in \mathcal{S}^{\textbf{r},1}(\tau_{N};\mathbb{R}^{d})$ such that there holds
\begin{equation}
\label{eq:2}
\begin{aligned}
\sum_{n=1}^{N}{\int_{I_{n}}}{(\dot{\mathcal{U}}(t),\varphi(t))\mathrm{d} t}&=\sum_{n=1}^{N}{\int_{I_{n}}{(\F(t,\mathcal{U}(t)),\varphi(t))}\mathrm{d} t}, \quad \forall \varphi \in \mathcal{S}^{\textbf{r}-1,}(\tau_{N};\mathbb{R}^{d}),\\
\mathcal{U}(0)&=u_{0}.
\end{aligned}
\end{equation}
It is noteworthy, that the discontinuous character of $\mathcal{S}^{r-1,0}(\tau_{N},\mathbb{R}^{d})$ allows us to choose discontinuous elements $\varphi \in \mathcal{S}^{\textbf{r}-1,1}(\tau_{N};\mathbb{R}^{d})$ and therefore problem \eqref{eq:2} decouples on each subinterval $I_{n}\subset I$ into
\begin{equation}
\label{eq:3}
\begin{aligned}
\int_{I_{n}}{(\dot{\mathcal{U}}(t),\varphi(t))\mathrm{d}t}&=\int_{I_{n}}{(\F(t,\mathcal{U}(t)),\varphi(t))\mathrm{d}t}, \quad \forall \varphi \in \PP^{r_{n}-1}(I_{n};\R^{d}),\\
\mathcal{U}|_{I_{n}}(t_{n-1})&=\mathcal{U}|_{I_{n-1}}(t_{n-1}),
\end{aligned}
\end{equation}
with $\mathcal{U}(0)=u_{0}$, i.e. the discretization can be interpreted as an implicit one-step scheme.

\subsection{Linearization}
For $u\in \PP^{r_{n}}(I_{n};\R^{d}) $ and $t\in (0,T]$ we introduce the operator $\G(t,u) \in \PP^{r_n-1}(I_{n};\R^d)'$ by 
\[
\left\langle \G(t,u)^{n},\varphi \right\rangle:=\int_{I_{n}}{(\dot{u}(t),\varphi(t))\mathrm{d} t}-\int_{I_{n}}{(\F(t,u(t)),\varphi(t))\mathrm{d} t}.
\]
In addition, let $\A:[0,T]\times \mathbb{R}^{d} \rightarrow \R^{d\times d}$ denote a map such that for $u\in\R^d$, there exists a uniform constant $C_{\A}$ with
\begin{equation}
\label{eq:CA}
\sup_{t\in [0,T]}{\norm{\A(t,u)x}}\leq C_{\A}\norm{x}, \quad x\in\R^d,
\end{equation}
and define further the operator 
\[\G_{\A}(t,u):\PP^{r_{n}}(I_{n};\R^{d})\rightarrow \PP^{r_{n}-1}(I_{n};\R^{d})'\] by
\[
\left\langle \G_{\A}(t,u)\delta,\varphi \right\rangle:=\int_{I_{n}}{(\dot{\delta}(t),\varphi(t))\mathrm{d} t}-\int_{I_{n}}{(\A(t,u(t))\delta(t),\varphi(t))\mathrm{d} t}.
\] 
Notice that if $\F$ is Fr\'echet differentiable in $u$ with derivative $\F_{u}$ with respect to $u$, then the G\^ateaux derivative in direction $\delta $ is given by
\[
\left\langle \G_{\F_{u}}(t,u)\delta,\varphi \right\rangle:=\int_{I_{n}}{(\dot{\delta}(t),\varphi(t))\mathrm{d} t}-\int_{I_{n}}{(\F_u(t,u(t))\delta(t),\varphi(t))\mathrm{d} t}.
\] 
Based on these definitions we introduce the following linearized \emph{hp}-cG iteration scheme at time node $t_{n}$: 
For $ j=0,1,2,\ldots$ and given $\mathcal{U}_{0}^n:=\mathcal{U}(t_{n-1})$ solve
\begin{equation}
-\left\langle \G_{\A}(t,\mathcal{U}_{j}^n)\delta_{j}^n,\varphi \right\rangle=\left\langle \G(t,\mathcal{U}_{j}^n),\varphi \right\rangle,
\end{equation}
for $\delta_{j}^{n}$ and compute the update
\[
\mathcal{U}_{j+1}^n=\mathcal{U}_{j}^n+\delta_{j}^n, \quad j\geq 0.
\]
In this note we freeze the second variable $\mathcal{U}_{j}^n$ in the operator $\G_{\A}(t,\mathcal{U}_{j}^n)$, i.e. we set $\G_{\A}(t,\mathcal{U}_{j}^n)\equiv \G_{\A}(t,\mathcal{U}_{0}^n)$ forall $j\geq 0$. This implies the following \emph{simplified} linearized \emph{hp}-cG iteration scheme at time node $t_{n}$: 

For $ j=0,1,2,\ldots$ and given $\mathcal{U}_{0}^n:=\mathcal{U}(t_{n-1})$ solve
\begin{equation}
\label{eq:linearized}
-\left\langle \G_{\A}(t,\mathcal{U}_{0}^n)\delta_{j}^n,\varphi \right\rangle=\left\langle \G(t,\mathcal{U}_{j}^n),\varphi \right\rangle,
\end{equation}
for $\delta_{j}^{n}$ and compute the update
\[
\mathcal{U}_{j+1}^n=\mathcal{U}_{j}^n+\delta_{j}^n, \quad j\geq 0.
\]
In case of $\A=\F_{u}$ the above iteration procedure \eqref{eq:linearized} can be interpreted as a \emph{simplified} Newton iteration scheme.


\section*{Convergence of the simplified linearized \upshape{\textbf{hp}}-\scshape{cG iteration scheme}}
\label{sec:3}

The aim of this Section is to show the existence of a solution for \eqref{eq:3}. Our strategy is to show that for sufficiently small time steps, the iteration scheme \eqref{eq:linearized} is well defined and convergent. Before we start, we need to collect some auxiliary results.

\begin{lemma}[Poincar\'e inequality]
Let $u\in H^{1}(I;\mathbb{R}^{n})$ with $u(a)=0\in\R^{d}$ and $I=(a,b) \subset \R, -\infty < a<b < \infty$. Then there holds the Poincar\'e inequality
\begin{equation}
\label{eq:poincare}
\norm{u}_{L^{2}(I;\R^d)}\leq kC_{\P}\norm{\dot{u}}_{L^{2}(I;\R^{n})},
\end{equation}
with $k=b-a$. The constant $C_{\P}>0$ is independent of $k$ and $u$.
\end{lemma}

\begin{proof}
This result is a direct consequence of the standard Poincar\'e inequality in $H^{1}(I;\R)$; (see \cite{braess2001finite} or \cite{evans2010partial}, for example).
\end{proof}
Let us further introduce the following set 
\begin{equation}
\label{eq:set1}
X_a^{r_{n}}:=\{x\in  \PP^{r_n}(I_{n};\R^{d})|x(t_{n-1})=a\}
\end{equation}
and notice that Poincar\'e's inequality \eqref{eq:poincare} holds true for all $x\in X_0^{r_n}$, i.e. we have
\[
\norm{x}_{L^{2}(I_{n};\R^d)}\leq k_{n}C_{\P}\norm{\dot{x}}_{L^{2}(I_{n};\R^d)}, \quad k_{n}=t_{n}-t_{n-1}. 
\] 
The following result addresses the invertibility of the operator $X_0^{r_{n}}\ni x\mapsto \G_{\A}(t,u)x$.
\begin{lemma}
\label{lemma:test}
Let $u \in \PP^{r_n}(I_n;\R^{d}), t \in (0,T]$. 
Then, if $ k_{n}<(C_{\A}C_{\P})^{-1}$ the operator
\[
\G_{\A}(t,u):X_0^{r_n}\rightarrow  \PP^{r_{n}-1}(I_n;\R^{d})' 
\]
is invertible on $\G_{\A}(t,u)(X_0^{r_n})\subset \PP^{r_{n}-1}(I_n;\R^{d})'$.
\end{lemma}

\begin{proof}
Since the linear map $x\mapsto \G_{\A}(t,u)x$ operates over a finite dimensional space, we show that its kernel is trivial. Suppose there exists $x\in X_0^{r_n}\setminus\{0\}$ such that 
\[
\left\langle \G_{\A}(t,u)x,\varphi \right\rangle=\int_{I_{n}}{(\dot{x}(t),\varphi(t))\mathrm{d} t}-\int_{I_{n}}{(\A(t,u(t))x(t),\varphi(t))\mathrm{d} t}=0,
\]
holds $\forall \varphi \in \mathcal{P}^{r_n-1}(I_n;\R^{d})$.
Choosing $\varphi=\dot{x}\in \PP^{r_n-1}(I_n;\R^{d})$, we conclude
\[
\int_{I_{n}}{\norm{\dot{x}}^{2}\mathrm{d}t}=\int_{I_{n}}{(\A(t,u)x,\dot{x})\mathrm{d}t}.
\]
Employing the Cauchy-Schwarz inequality we get
\[
\begin{aligned}
\int_{I_{n}}{\norm{\dot{x}}^{2}\mathrm{d}t}&=\int_{I_{n}}{(\A(t,u)x,\dot{x})\mathrm{d}t}\\
&\leq \int_{I_{n}}{\abs{(\A(t,u)x,\dot{x})}\mathrm{d}t}\\
&\leq \int_{I_{n}}{\norm{\A(t,u)x}\norm{\dot{x}}\mathrm{d}t}\\
&\leq C_{\A} \int_{I_{n}}{\norm{x}\norm{\dot{x}}\mathrm{d}t}\\
&\leq C_{\A} \left(\int_{I_{n}}{\norm{x}^2\mathrm{d}t}\right)^{\nicefrac{1}{2}}\left(\int_{I_{n}}{\norm{\dot{x}}^2\mathrm{d}t}\right)^{\nicefrac{1}{2}}.
\end{aligned}
\]
Hence, the above estimate implies
\[
\int_{I_{n}}{\norm{\dot{x}}^{2}\mathrm{d}t}\leq C_{\A}^2 \int_{I_{n}}{\norm{x}^2\mathrm{d}t}.
\]
Since $x\in X_0^{r_{n}}$ we can invoke the Poincar\'e inequality and obtain, together with our assumption $(C_{\P}C_{\A})^{-1}>k_{n}$, the following contradiction
\[
\norm{\dot{x}}^{2}_{L^{2}(I_{n};\R^d)}\leq C_{\A}^2k_{n}^2C_{\P}^2\norm{\dot{x}}_{L^2(I_n;\R^d)}^2<\norm{\dot{x}}_{L^2(I_n;\R^d)}^2.
\]
\end{proof}

\begin{remark}
Note that if $\A\equiv 0$ the operator $\G_{\A}(t,u)$ is invertible without any condition on $k_n$.
Indeed, let $x\in X_0^{r_{n}}$ with
\[
0=\left\langle \G_{\A}(t,u)x(t),\varphi(t) \right\rangle, \quad  \forall \varphi \in \PP^{r_{n}-1}(I_n;\R^d),
\]
i.e. for $ \dot{x}=\varphi$ there holds
\[
\int_{I_{n}}{\norm{\dot{x}}^{2}\mathrm{d}t}=0,
\]
which implies $\norm{x}=0$ and therefore $x=0$.
\end{remark}
Next, for $u\in \mathcal{P}^{r_{n}}(I_n;\R^d)$ we introduce the following map 
\begin{equation}
\label{eq:fixed-point-map}
g: X_a^{r_{n}} \rightarrow  X_a^{r_{n}}
\end{equation}
defined by
\[
-\left\langle \G_{\A}(t,u)(g(x)-x),\varphi \right\rangle=\left\langle \G(t,x),\varphi \right\rangle \quad \forall \varphi \in \PP^{r_{n}-1}(I_n;\R^d).
\]
Notice that $g$ is well defined by the above Lemma \ref{lemma:test}, i.e.
\[
g(x)=x-[\G_{\A}(t,u)]^{-1}\G(t,x) \in X_{a}^{r_{n}}.
\]
The next result shows that $g$ is a Lipschitz continuous map. 
\begin{lemma}
If $(C_{\A}C_{\P})^{-1}>k_n$, then there holds
\begin{equation}
\label{eq:final}
\norm{g(x)-g(y)}_{L^2(I_{n};\R^n)}\leq \frac{C_{\A}+L}{(k_{n}C_{\P})^{-1}-C_{\A}}\norm{x-y}_{L^2(I_{n};\R^n)}, \quad x,y \in X_{a}^{r_{n}}.
\end{equation}
\end{lemma}

\begin{proof}
Let $x,y\in X_{a}^{r_{n}}$, set $s:=g(x)-g(y) \in X_{0}^{r_{n}}$ and notice that $\dot{s}\in \PP^{r_{n}-1}(I_n;\R^d)$. By definition of $g$ there holds
\[
\left\langle \G_{\A}(t,u)s,\dot{s} \right\rangle=\left\langle \G_{\A}(t,u)(x-y),\dot{s} \right\rangle-\left\langle \G(t,x)-\G(t,y),\dot{s} \right\rangle
\]
with
\[
\begin{aligned}
&\left\langle \G_{\A}(t,u)(x-y),\dot{s} \right\rangle-\left\langle \G(t,x)-\G(t,y),\dot{s} \right\rangle\\
=&\int_{I_{n}}{(\F(t,x)-\F(t,y),\dot{s})\mathrm{d}t}-\int_{I_{n}}{(\A(t,u)(x-y),\dot{s})\mathrm{d}t},
\end{aligned}
\]
and
\[
\left\langle \G_{\A}(t,u)s,\dot{s} \right\rangle=\int_{I_{n}}{\norm{\dot{s}}^2\mathrm{d}t}-\int_{I_{n}}{(\A(t,u)s,\dot{s})\mathrm{d}t}.
\]
Hence we arrive at
\begin{equation}
\label{eq:12}
\int_{I_{n}}{\norm{\dot{s}}^2\mathrm{d}t}=\int_{I_{n}}{(\A(t,u)(s-(x-y)),\dot{s})\mathrm{d}t}+\int_{I_{n}}{(\F(t,x)-\F(t,y),\dot{s})\mathrm{d}t}.
\end{equation}
By \eqref{eq:12} and the Cauchy-Schwarz inequality we obtain 
\[
\begin{aligned}
\int_{I_{n}}{\norm{\dot{s}}^2\mathrm{d}t}\leq & C_{\A}\int_{I_{n}}{\norm{s-(x-y)	}\norm{\dot{s}}\mathrm{d}t}+L\int_{I_{n}}{\norm{x-y}{\norm{\dot{s}}}\mathrm{d}t}\\
\leq & C_{\A}\int_{I_{n}}{\norm{s}\norm{\dot{s}}\mathrm{d}t}+(C_{\A}+L)\int_{I_{n}}{\norm{x-y}{\norm{\dot{s}}}\mathrm{d}t}\\
\leq &\left( C_{\A}\left(\int_{I_{n}}{\norm{s}}^2\mathrm{d}t\right)^{\nicefrac{1}{2}}+(C_{\A}+L)\left(\int_{I_{n}}{\norm{x-y}^2\mathrm{d}t}\right)^{\nicefrac{1}{2}}\right)\left(\int_{I_{n}}{\norm{\dot{s}}}^2\mathrm{d}t\right)^{\nicefrac{1}{2}},
\end{aligned}
\]
i.e., there holds
\[
\norm{\dot{s}}_{L^{2}(I_{n};\R^d)}\leq C_{\A}\norm{s}_{L^{2}(I_{n};\R^d)}+(C_{\A}+L)\norm{x-y}_{L^2(I_n;\R^d)}.
\]
Using that $s\in X^{r_{n}}$ and invoking again the Poincar\'e inequality we end up with
\begin{equation}
\label{eq:22}
(k_n C_{\P})^{-1}\norm{s}_{L^{2}(I_{n};\R^d)}\leq  C_{\A}\norm{s}_{L^{2}(I_{n};\R^d)}+(C_{\A}+L)\norm{x-y}_{L^2(I_n;\R^d)}.
\end{equation}
Solving \eqref{eq:22} for $\norm{\dot{s}}_{L^{2}(I_n,\mathbb{R}^d)}$ by using the assumption $(k_n C_{\P})^{-1}>C_{\A}$, we obtain \eqref{eq:final}.
\end{proof}

\begin{remark}
Notice that for $\A\equiv 0$, estimate \eqref{eq:final} is simply
\[
\norm{g(x)-g(y)}_{L^2(I_{n};\R^n)}\leq Lk_{n}C_{\P}\norm{x-y}_{L^2(I_{n};\R^n)}, \quad x,y \in \mathcal{P}^{r_n}(I_n,\mathbb{R}^{d}).
\] 
\end{remark}
Using the above results we are now ready to prove the main result of this note:
\begin{theorem}
\label{theo:main}
Suppose there holds 
\begin{equation}
\label{eq:main-assum}
k< \frac{1}{(2C_{\A}+L)C_{\P}},
\end{equation}
and assume further that $\F$ is Lipschitz continuous in the second variable with Lipschitz constant $L$. Then, the cG method \eqref{eq:2} admits a unique solution $\mathcal{U}\in \mathcal{S}^{\textbf{r},1}(\tau_{N};\R^{d})$.  
\end{theorem}

\begin{proof}
On each interval $I_{n}$ there holds by assumption \eqref{eq:main-assum} 
\[
k_{n}\leq k< \frac{1}{(2C_{\A}+L)C_{\P}}<\frac{1}{C_\A C_{\P}}.
\] 
Therefore, for each time step the map $g:X_{a}^{r_{n}}\rightarrow X_{a}^{r_{n}}$ with $a=\mathcal{U}(t_{n-1})$ is well defined by Lemma \ref{lemma:test}. In addition, for each time step the map $g$ is a contraction on $X_a^{r_n}$ by \eqref{eq:main-assum}, i.e. there exists a unique $\mathcal{U} \in X_{a}^{r_{n}}$ with $g(\mathcal{U})=\mathcal{U}$ which is the desired zero for $\G(t,\mathcal{U})$.
\end{proof}

\begin{remark}
This solution can be obtained iteratively by employing the iterative scheme from \eqref{eq:linearized}.
\end{remark}


\section{Numerical Experiments}
\label{sec:4}

In this Section we show some numerical experiments illustrating the theoretical convergence results from Section \ref{sec:3}. In doing so we provide some apriori error results given in \cite{Wihler:05}. As a preparation, we need to define the number of degrees of freedom DOF for the cG method, which is simply the sum of all DOF's on each subinterval $I_{n}$ that are needed to compute the numerical solution $\mathcal{U}$ on $I$. Thus 
\[
\text{DOF}=\text{dim}\left(\mathcal{S}^{\textbf{r}-1,0}(\tau_{N},\mathbb{R}^{d})\right)=d\sum_{n=1}^{N}{r_{n}}.
\]
Here we fix the polynomial degree, i.e. we set $r_{n}=r$ and therefore $\text{DOF}=d\cdot N \cdot r$. In addition, we assume that the solution $u$ of \eqref{eq:initial} belongs to $H^{s+1}(I,\mathbb{R}^{d}), s\geq 0$, and that the partition of $I$ is quasi-uniform. Then the apriori error results with respect to the $L^2$-norm is given in \cite{Wihler:05} through 
\begin{equation}
\label{eq:aprioriL2}
\norm{u-\mathcal{U}}_{L^{2}(I,\mathbb{R}^{d})}\leq C \text{DOF}^{-\min{(s,r)}-1}.
\end{equation}
The constant $C$ depends on $L,T$ and is independent of $r$ and $\mathcal{T}_{N}$.
Let us further point out the following two aspects arising from the apriori error result \eqref{eq:aprioriL2}:
\subsubsection*{h-version}	The h-version of the cG scheme means that convergence is implied by increasing the number of time steps at a fixed approximation order $r_n=r\geq 1$ on each interval $I_n$. In case that the solution $u$ of \eqref{eq:initial} is analytic and therefore  $s$ is large, the error estimate \eqref{eq:aprioriL2} then reads
\[
\norm{u-\mathcal{U}}_{L^{2}(I,\mathbb{R}^{d})}\leq C \text{DOF}^{-r-1}.
\] 
For $h\to 0$ or equivalently $N\to \infty$, we see that the rate of convergence with respect to the $L^2$-norm is $r+1$.
\subsubsection*{p-version} For the $p$-version of the cG method we keep the time partition fixed but let the approximation order $r_{n}$ be variable. Thus, convergence is obtained by increasing the approximation order. In addition, it can be shown that for analytic solutions the $p$-version admits high order convergence (even exponentially with respect to $r$). We refer here to \cite{Wihler:05} and \cite{schwab1998p} for further details.
\begin{example}
\label{ex:1}

The first example is given by the initial value problem
\begin{equation}
\label{eq:ex1}
\begin{cases}
\dot{u}(t)&=-\sin(u(t)), t\in (0,1],\\
u(0)&=\frac{\pi}{2}.
\end{cases}
\end{equation} 
The exact solution is given by $u(t)=2\arctan(\mathrm{e} ^{-t})$. We test the simplified Newton scheme, i.e. we consider \eqref{eq:linearized} using $\A \equiv \F_{u}$. Notice that $C_{\A}=L=1$ and therefore we can expect convergence by Theorem \ref{theo:main}. First we present the $h$-version of the cG method in Figure~\ref{fig:h-performance1}.The numerical test was obtained by bisection of the time interval $I=[0,1]$. In Figure~\ref{fig:h-performance1} we depict the decay of the $L^{2}$-error from where we can clearly see the convergence order $r+1$ according to the error estimate given in \eqref{eq:aprioriL2}. Table~\ref{tab:1} shows the $L^{2}$-error as well as the convergence order $r+1$ and the number of elements. In addition, Tables~\ref{tab:2} and \ref{tab:3} show also the number of iterations compared with the number of iterations when solving \eqref{eq:linearized} with $\A \equiv 0$ which is simply the standard Banach-fixed point iteration procedure. As it can be seen, the simplified Newton scheme needs a significant lower number of iterations in order to accomplish the $L^{2}$-error given in Table~\ref{tab:1}. Moreover, in Figure~\ref{fig:p-performance1} we depict the $p$-version of the cG method from where we observe that the convergence rates are even exponential. Here we choose a fixed partition of $[0,1]$ using $2,5,10,20$ and $50$ elements, i.e. we increase the approximation order on a fixed partition of $[0,1]$.   

\begin{table}[tbp]
\caption{Example 1: The performance data of the $h$-version of the cG method with use of the \emph{simplified} Newton method.}
\resizebox{\columnwidth}{!}{%
\begin{tabular}{c|c|c|c|c|c|c|c|c|c|c}
               \#el. & $r=1$  & Ord. & $r=2$ & Ord. & $r=3$  & Ord. & $r=4$ & Ord. & $r=5$ & Ord. \\
\hline
 1 &   2.94e-02 & - &  2.83e-3  & - & 4.51e-4 & - & 8.54e-6 & - & 6.78e-6 & - \\
 2 &   7.54e-03 & 1.97 & 4.39e-4 &  2.72 & 2.41e-5 & 4.23 & 1.30e-6 & 2.71 & 8.76e-8 & 6.27 \\
 4 &  1.90e-03 & 1.99 & 5.48e-5  & 2.97 &  1.51e-6  & 3.99 & 4.18e-8 & 4.96 & 1.34e-9 & 6.02 \\
 8 &  4.77e-04  & 2.00 & 6.89e-6 & 2.99 & 9.48e-8 & 4.00 & 1.31e-9 & 4.99 & 2.11e-11 & 5.99 \\
16 & 1.19e-04 & 2.00 &  8.62e-7 & 3.00 & 5.93e-9 & 4.00 & 4.10e-11 & 5.00 & 3.31e-14 & 6.00 \\
 32 & 2.98e-05  & 2.00 & 1.07e-7  & 3.00 & 3.70e-10 & 4.00 & 1.28e-12 & 5.00 & -- & -- \\
 64 & 7.46e-06 & 2.00 &  1.35e-8 & 3.00 & 2.31e-11 & 4.00 & 4.01e-14 & 5.00 & -- & -- \\
 128 & 1.86e-06 & 2.00 & 1.68e-9 & 3.00 & 1.45e-12 & 4.00 & -- & -- & -- & -- \\
 256 &  4.66e-07 & 2.00 & 2.10e-10 & 3.00 & 8.80e-14 & 4.00 & -- & -- & -- & -- \\
 512 & 1.17e-07 & 2.00 &  2.63e-11 & 3.00 & -- & -- & -- & -- & -- & -- \\
 1024 &  2.91e-08 & 2.00 & 3.29e-12 & 3.00	 & -- & -- & -- & -- & -- & -- \\
 2048 & 7.28e-09  & 2.00 & -- & -- & -- & -- & -- & -- & -- & -- \\
 4096 & 1.82e-09 & 2.00 & -- & -- & -- & -- & -- & -- & -- & -- \\
 8192 & 4.55e-10 & 2.00 & -- & -- & -- & -- & -- & -- & -- & -- 
\end{tabular}
}
\label{tab:1}
\end{table}

\vspace{0.3cm}

\begin{table}[tbp]
\caption{Example 1: The performance data of the $h$-version of the cG method including the number of iterations (with use of the \emph{simplified} Newton method).}
\resizebox{\columnwidth}{!}{%
\begin{tabular}{c|c|c|c|c|c|c|c|c|c|c}
               \#el. & $r=1$  & It. & $r=2$ & It. & $r=3$  & It. & $r=4$ & It. & $r=5$ & It. \\
\hline
 1 &   2.94e-02 & 27 & 2.83e-3 & 19 & 4.51e-4 & 17 & 8.54e-6  & 16 & 6.78e-6	 & 15 \\
 2 &   7.54e-03 & 28 & 4.39e-4 & 23 & 2.41e-5 & 22 & 1.30e-6 & 22 & 8.76e-8 & 21 \\
 4 &  1.90e-03 & 39 & 5.48e-5 & 35 & 1.51e-6 & 33 &  4.18e-8 & 34 & 1.34e-9 & 34 \\
 8 &  4.77e-04  & 62 & 6.89e-6 & 56 & 9.48e-8 & 56 & 1.31e-9 & 56 & 2.11e-11 & 56 \\
16 & 1.19e-04 & 96 & 8.62e-7 & 96 & 5.93e-9 & 96 & 4.10e-11 & 96 & 3.31e-14 & 96 \\
 32 & 2.98e-05  & 185 & 1.07e-7 & 177 & 3.70e-10 & 173 &  1.28e-12 & 173 & -- & -- \\
 64 & 7.46e-06 & 320 & 1.35e-8 & 320 & 2.31e-11 & 319 & 4.01e-14 & 319 & -- & -- \\
 128 & 1.86e-06 & 628 & 1.68e-9 & 624 & 1.45e-12  & 627 & - & -- & -- & -- \\
 256 &  4.66e-07 & 1024 & 2.10e-10 & 1024 & 8.80e-14 & 1024 & -- & -- & -- & -- \\
 512 & 1.17e-07 & 2048 & 2.63e-11 & 2047 & -- & -- & -- & -- & -- & -- \\
 1024 &  2.91e-08 & 4089 & 3.29e-12 & 4090 & -- & -- & -- & -- & -- & -- \\
 2048 & 7.28e-09  & 7970 & -- & -- & -- & -- & -- & -- & -- & -- \\
 4096 & 1.82e-09 & 12288 & -- & -- & -- & -- & -- & -- & -- & -- \\
 8192 & 4.55e-10 & 24579 & -- & -- & -- & -- & -- & -- & -- & -- 
\end{tabular}
}
\label{tab:2}
\end{table}

\begin{table}[tbp]
\caption{Example 1: The performance data of the $h$-version of the cG method including the number of iterations (with use of the Banach fixed point iteration procedure).}
\resizebox{\columnwidth}{!}{%
\begin{tabular}{c|c|c|c|c|c|c|c|c|c|c}
               \#el. & $r=1$  & It. & $r=2$ & It. & $r=3$  & It. & $r=4$ & It. & $r=5$ & It. \\
\hline
 1 &   2.94e-02 & 27 & 2.83e-3 & 19 & 4.51e-4 & 18 & 8.54e-6  & 16 & 6.78e-6	 & 15 \\
 2 &   7.54e-03 & 35 & 4.39e-4 & 28 & 2.41e-5 & 27 & 1.30e-6 & 24 & 8.76e-8 & 24 \\
 4 &  1.90e-03 & 51 & 5.48e-5 & 44 & 1.51e-6 & 45 &  4.18e-8 & 40 & 1.34e-9 & 40 \\
 8 &  4.77e-04  & 83 & 6.89e-6 & 74 & 9.48e-8 & 77 & 1.31e-9 & 69 & 2.11e-11 & 69 \\
16 & 1.19e-04 & 138 & 8.62e-7 & 127 & 5.93e-9 & 135 & 4.10e-11 & 120 & 3.31e-14 & 120 \\
 32 & 2.98e-05  & 239 & 1.07e-7 & 222 & 3.70e-10 & 242 &  1.28e-12 & 216 & -- & -- \\
 64 & 7.46e-06 & 425 & 1.35e-8 & 393 & 2.31e-11 & 440 & 4.01e-14 & 398 & -- & -- \\
 128 & 1.86e-06 & 753 & 1.68e-9 & 731 & 1.45e-12  & 848 & - & -- & -- & -- \\
 256 &  4.66e-07 & 1427 & 2.10e-10 & 1315 & 8.80e-14 & 1523 & -- & -- & -- & -- \\
 512 & 1.17e-07 & 2517 & 2.63e-11 & 2498 & -- & -- & -- & -- & -- & -- \\
 1024 &  2.91e-08 & 4903 & 3.29e-12 & 4800 & -- & -- & -- & -- & -- & -- \\
 2048 & 7.28e-09  & 8922 & -- & -- & -- & -- & -- & -- & -- & -- \\
 4096 & 1.82e-09 & 16316 & -- & -- & -- & -- & -- & -- & -- & -- \\
 8192 & 4.55e-10 & 32384 & -- & -- & -- & -- & -- & -- & -- & -- 
\end{tabular}
}
\label{tab:3}
\end{table}

\begin{figure}
\includegraphics[width=0.7\textwidth]{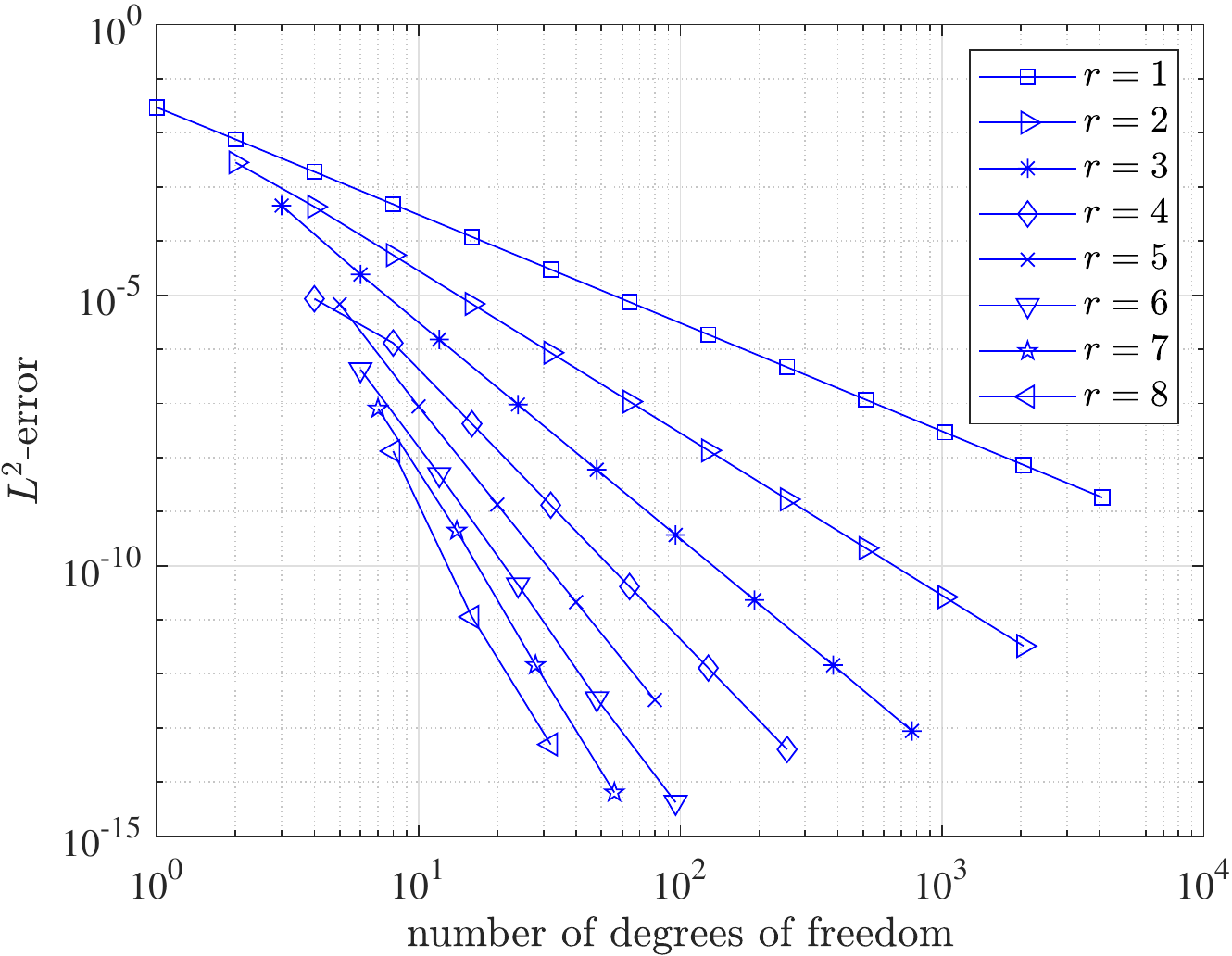}
\caption{Example~\ref{ex:1}: The $h$-version of the cG method (with use of the \emph{simplified} Newton method).}
\label{fig:h-performance1}
\end{figure}

\begin{figure}
\includegraphics[width=0.7\textwidth]{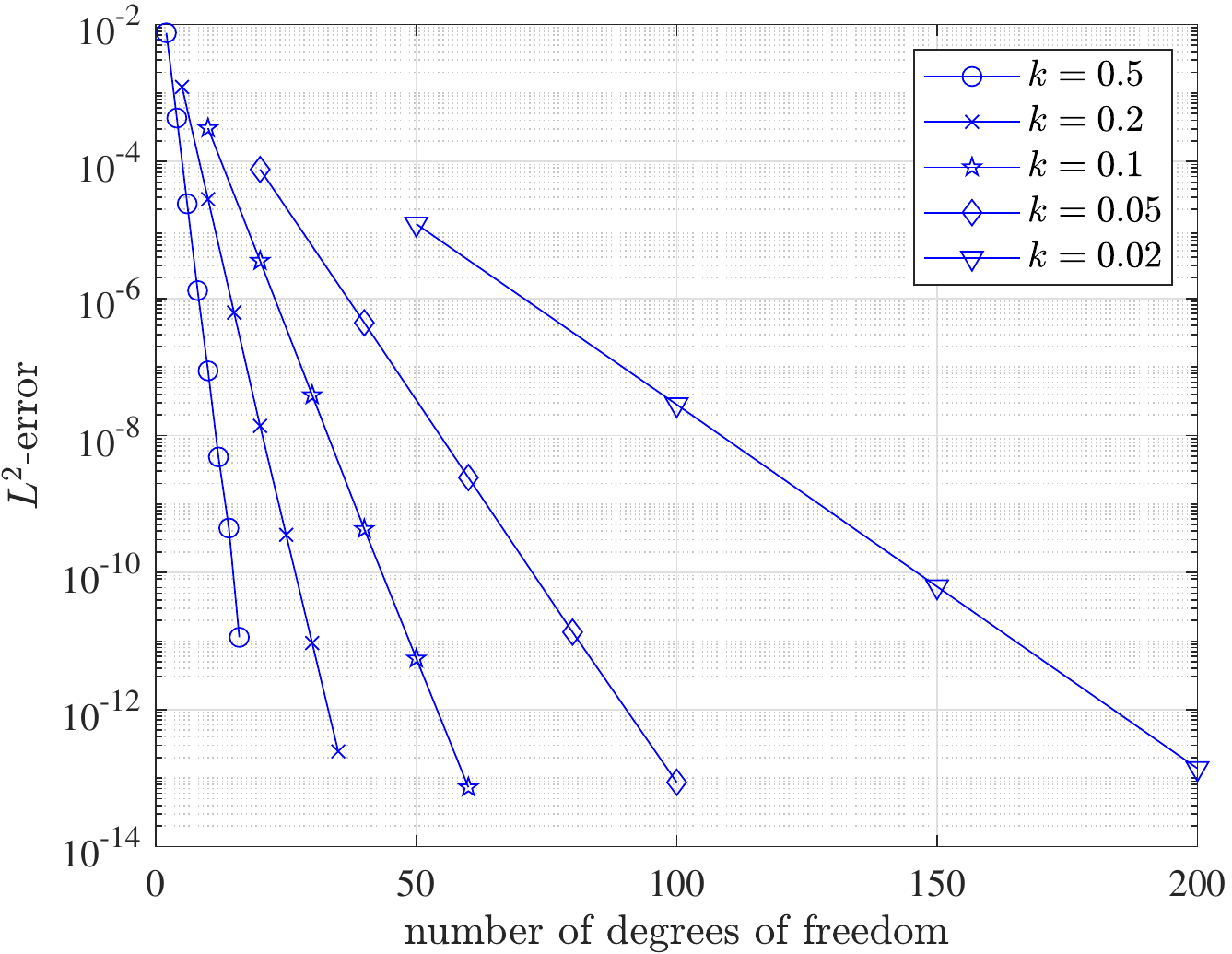}
\caption{Example~\ref{ex:1}: The $p$-version of the cG method (with use of the \emph{simplified} Newton method).}
\label{fig:p-performance1}
\end{figure}

\end{example}

\begin{example}
\label{ex:2}
The second example is given by
\begin{equation}
\label{eq:ex2}
\begin{cases}
\dot{u}(t)&=-2tu(t)^2, t\in (0,1],\\
u(0)&=1.
\end{cases}
\end{equation} 
The exact solution is given by $u(t)=\frac{1}{1+t^2}$. Notice, that \eqref{eq:Lipschitz} and \eqref{eq:CA} are not satisfied. However, we again depict in Figure~\ref{fig:h-performance2} 
the $h$-version of the cG method (again by bisection of the time interval $I=[0,1]$) from where we can clearly see the convergence order $r+1$according to the error estimate given in \eqref{eq:aprioriL2}. Furthermore, in Figure~\ref{fig:p-performance2} we see the $p$-version of the cG method from where we observe that the convergence rates are again exponential (again we choose a fixed partition of $[0,1]$ using $2,5,10,20$ and $50$ elements, i.e. we increase the approximation order on a fixed partition of $[0,1]$).

\begin{figure}
\includegraphics[width=0.7\textwidth]{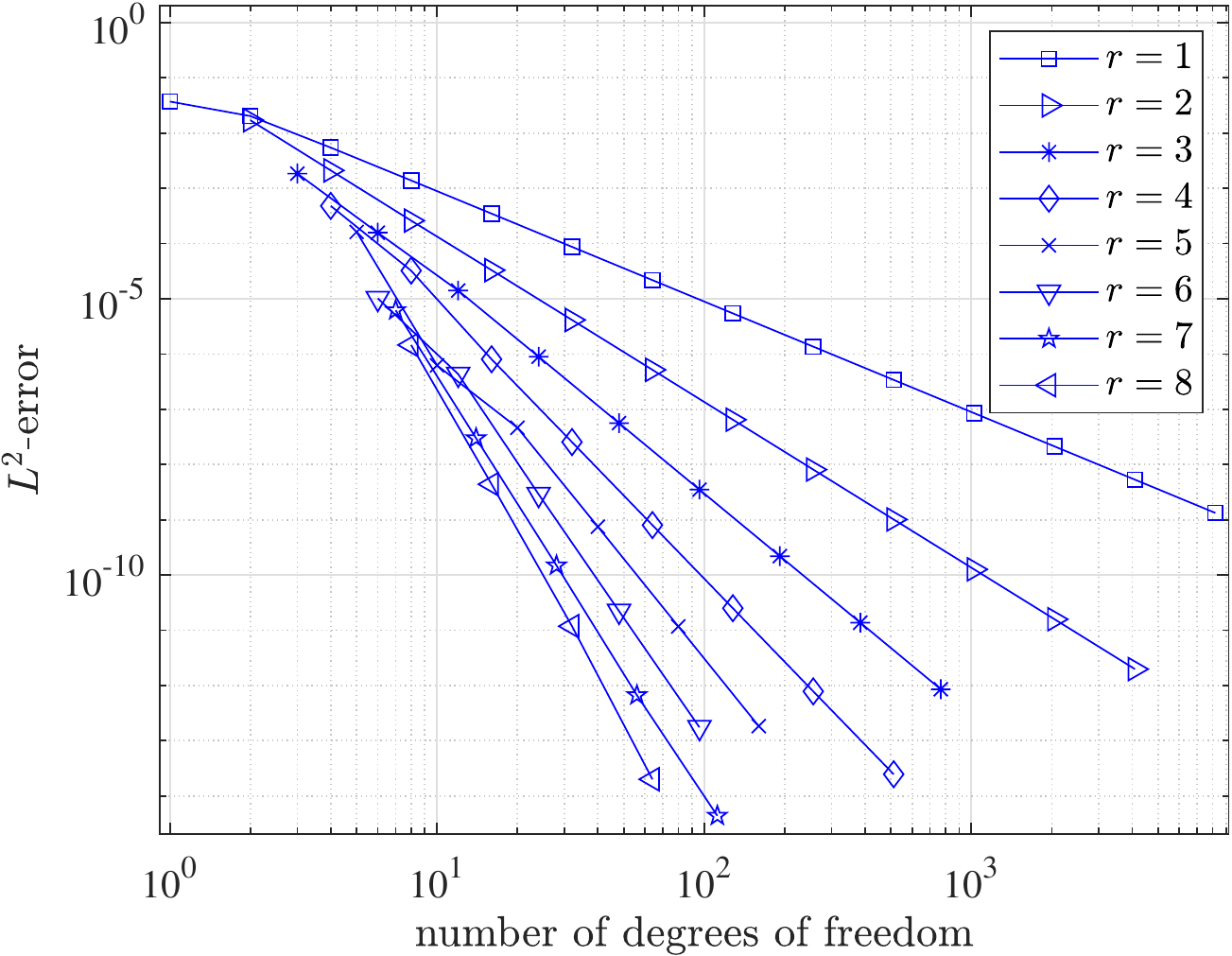}
\caption{Example~\ref{ex:2}: The $h$-version of the cG method (with use of the \emph{simplified} Newton method).}
\label{fig:h-performance2}
\end{figure}

\begin{figure}
\includegraphics[width=0.7\textwidth]{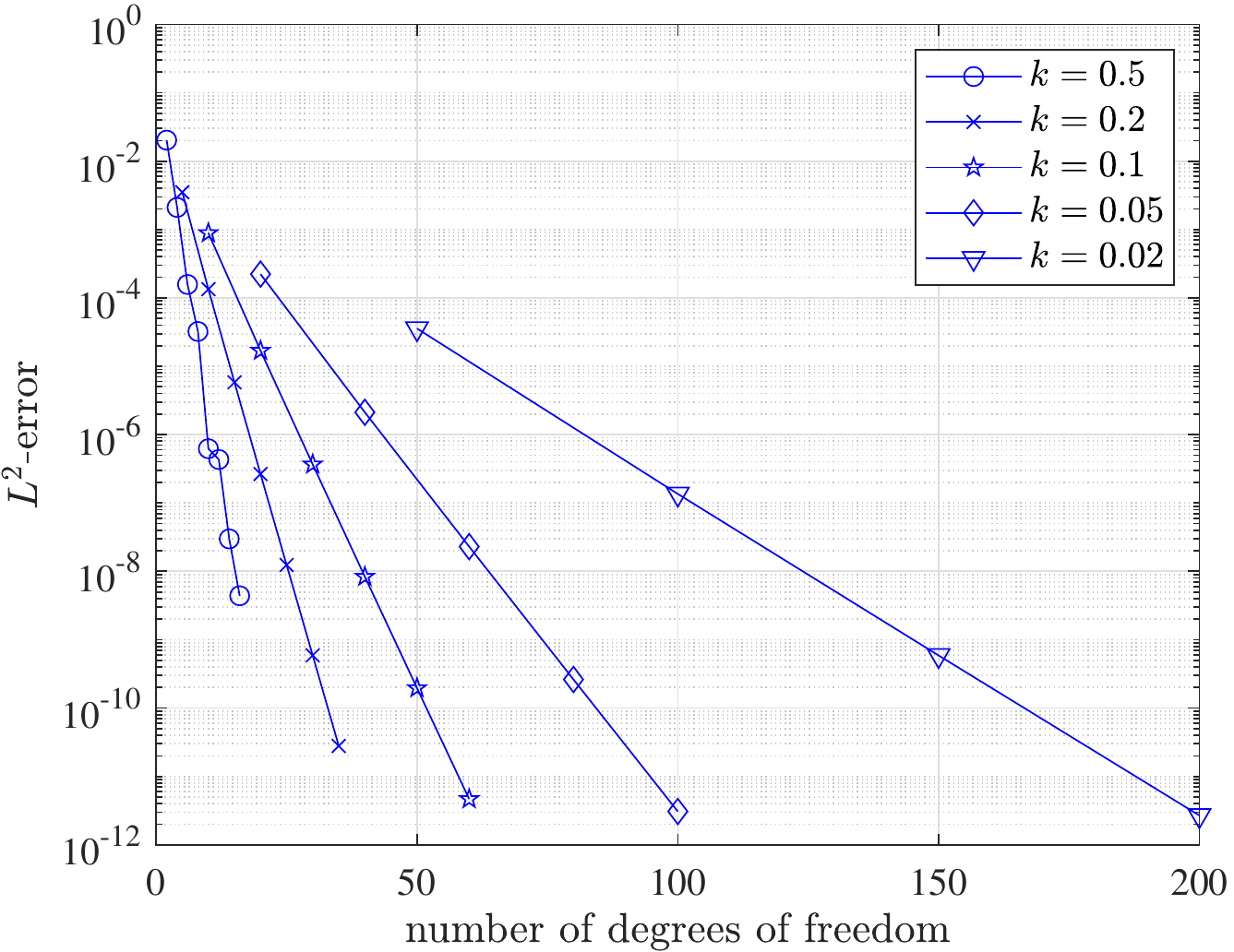}
\caption{Example~\ref{ex:2}: The $p$-version of the cG method (with use of the \emph{simplified} Newton method).}
\label{fig:p-performance2}
\end{figure}

\end{example}

\begin{example}
\label{ex:3}
We finally consider the following nonlinear initial value problems:

Given the initial data $u_0=(\nicefrac{\pi}{4},1)$ we seek $u=(u_{1},u_{2})$ such that for $t\in (0,1]$ there holds
\begin{equation}
\label{eq:system}
\begin{aligned}
\dot{u}_{1}(t)&=-\frac{u_2}{1+u_{2}^2}, \\
\dot{u}_{2}(t)&=-\tan(u_{1}).
\end{aligned}
\end{equation}
We use the exact solution $u(t)=(\arctan(\mathrm{e}^{-t}),\mathrm{e}^{-t})$ as reference solution. We remark that, although the assumptions \eqref{eq:Lipschitz} and \eqref{eq:CA} are again not necessarily satisfied for this problem, our approach still delivers good results as can be seen from the Figures~\ref{fig:h-performance_syst} and~\ref{fig:p-performance_syst}.

\begin{figure}
\includegraphics[width=0.7\textwidth]{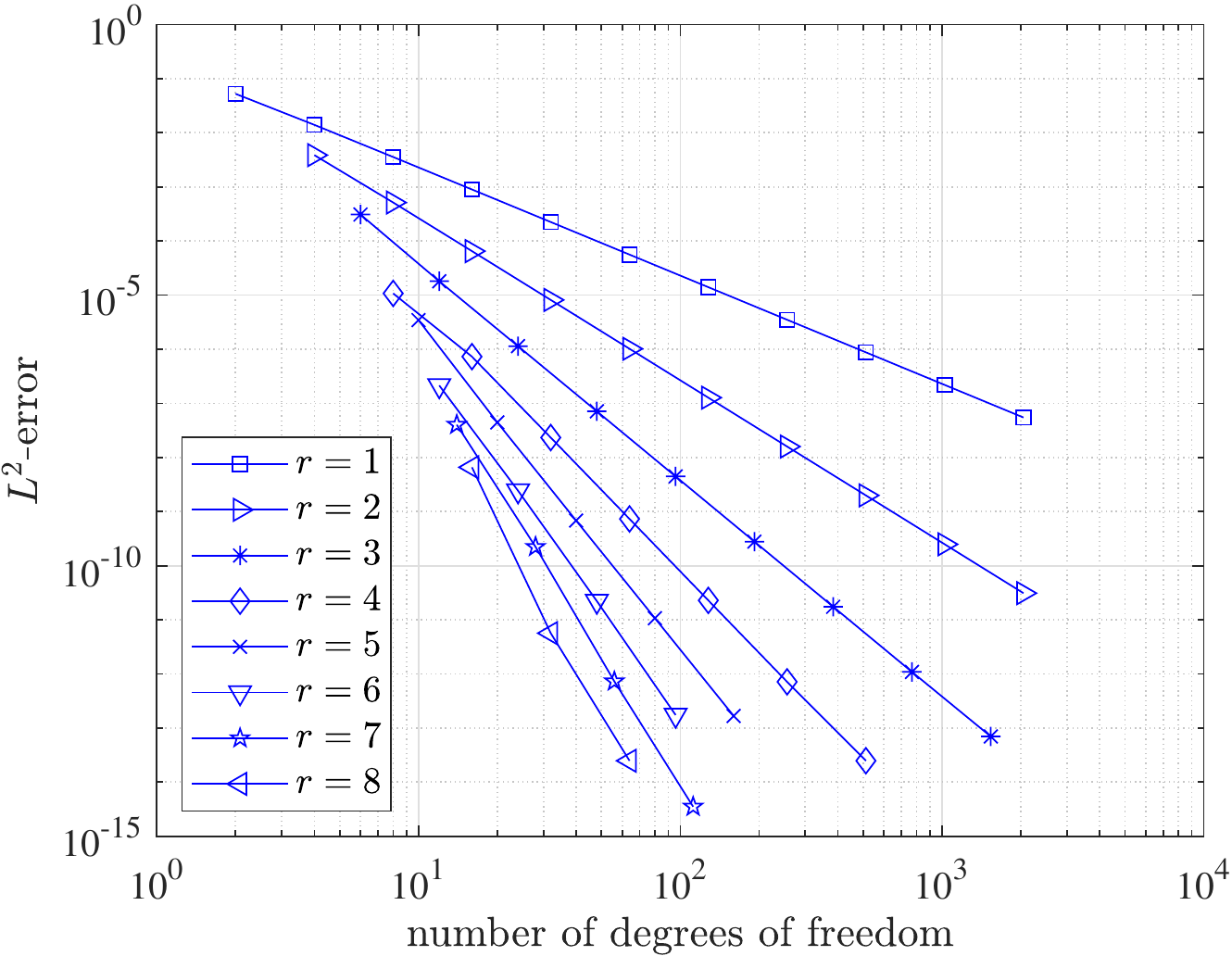}
\caption{Example~\ref{ex:3}: The $h$-version of the cG method (with use of the \emph{simplified} Newton method).}
\label{fig:h-performance_syst}
\end{figure}

\begin{figure}
\includegraphics[width=0.7\textwidth]{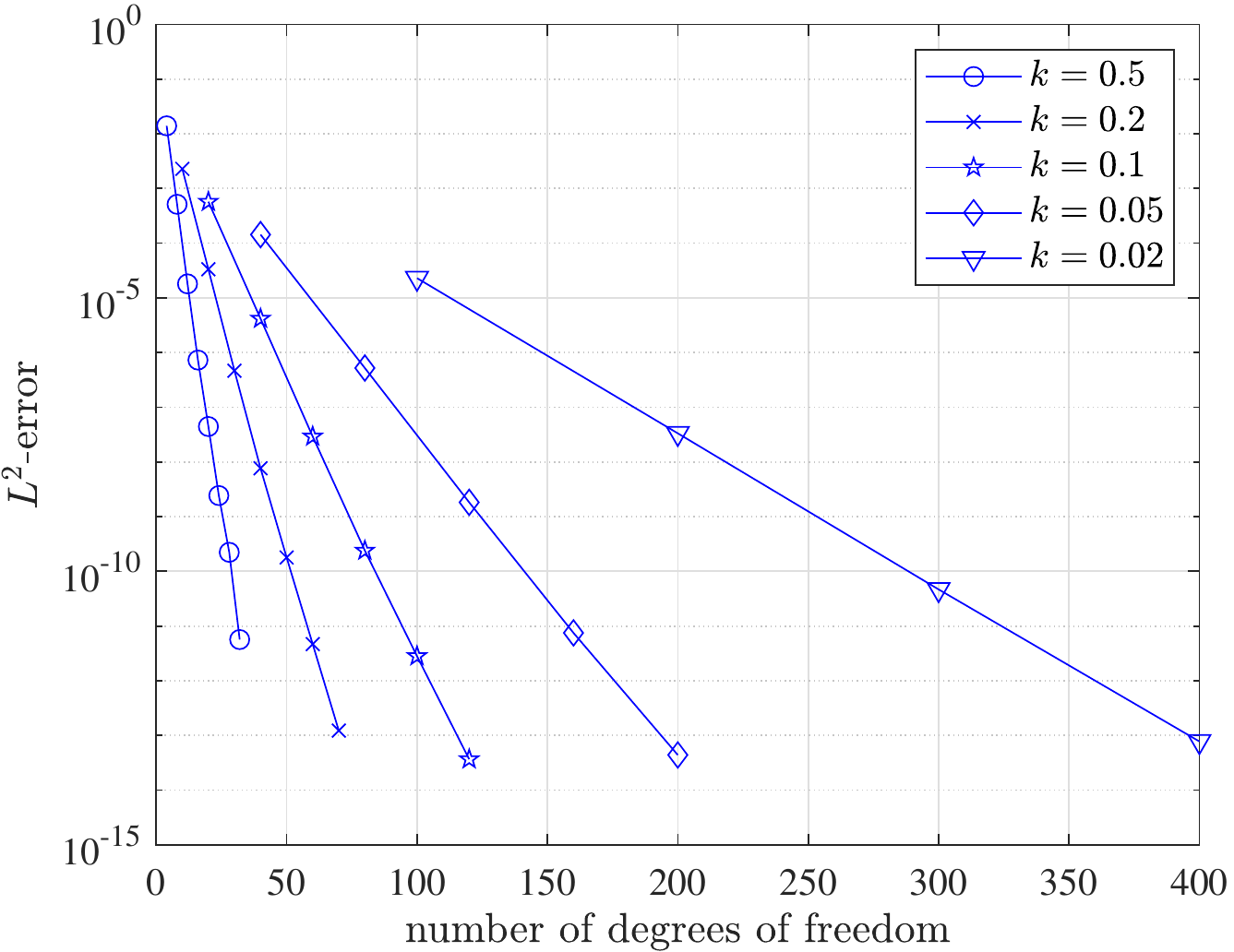}
\caption{Example~\ref{ex:3}: The $p$-version of the cG method (with use of the \emph{simplified} Newton method).}
\label{fig:p-performance_syst}
\end{figure}

\end{example}

\section{Conclusions}
\label{sec:concl}
The aim of this note was the numerical solution of initial value problems by means of the continuous Galerkin method for the discretization of the underlying nonlinear problem. We further have shown the existence of the discrete solution under reasonable assumptions. In addition, we have proved that for sufficiently small time steps, the linearized continuous Galerkin scheme admits a unique solution that can be obtained iteratively. Moreover, we have tested the proposed iteration scheme on a series of numerical examples. Our numerical experiments clearly illustrate the ability of our approach. Specifically, the \emph{simplified} Newton iteration scheme was able to significantly reduce the computational amount by means of the number of iterations.

\bibliographystyle{amsplain}
\bibliography{references}
\end{document}